\documentclass[12pt,a4paper]{article}
\usepackage{mathrsfs}
\usepackage{epsfig, graphicx}
\usepackage{latexsym,amsfonts,amsbsy,amssymb}
\usepackage{amsmath,amsthm}
\usepackage{color}
\usepackage[colorlinks, citecolor=blue]{hyperref}
\makeatletter 

\@addtoreset{equation}{section}
\makeatother 
\allowdisplaybreaks 

\newtheorem{theorem}{Theorem}[section]
\newtheorem{lemma}{Lemma}[section]
\newtheorem{corollary}{Corollary}[section]
\newtheorem{remark}{Remark}[section]
\newtheorem{algorithm}{Algorithm}[section]

\begin{document}
\title{A Cascadic Multigrid Method for GPE
Problem\footnote{This work is supported in part by the National
Science Foundation of China (NSFC 91330202, 11371026, 11001259,
11031006, 2011CB309703), the National Center for Mathematics
and Interdisciplinary Science of CAS Chinese Academy of Sciences
 and the President Foundation of Academy of Mathematics and
 Systems Science-Chinese Academy of Sciences.}
}
\author{Xiaole Han\footnote{LSEC, ICMSEC, Academy of Mathematics and Systems
Science, Chinese Academy of Sciences,  Beijing 100190, China
(hanxiaole@lsec.cc.ac.cn)}\ \ \
 and\ \ Hehu Xie\footnote{LSEC, ICMSEC,
Academy of Mathematics and Systems Science, Chinese Academy of
Sciences,  Beijing 100190, China (hhxie@lsec.cc.ac.cn)}
}
\date{}
\maketitle

\begin{abstract}
{
A cascadic multigrid method is proposed for the GPE problem based on the
multilevel correction scheme. With this new scheme, the ground state
eigenvalue problem on the finest space
can be solved by smoothing steps on a series of multilevel finite element spaces and some 
nonlinear eigenvalue problem solving on a very low-dimensional space.
 Choosing the appropriate sequence of finite element spaces and the number of
 smoothing steps, the optimal convergence rate with the optimal computational
 work can be arrived. Some numerical experiments are presented to validate our theoretical analysis.
}
\vskip0.3cm {\bf Keywords.} Bose-Einstein condensation; Gross-Pitaevskii equation; multilevel correction; 
 cascadic multigrid; nonlinear eigenvalue problem; finite element method.

\vskip0.2cm {\bf AMS subject classifications.} 65N30, 65N25, 65L15,
65B99.
\end{abstract}

\section{Introduction}\label{sec;introduction}
The aim of this paper is to design a cascadic type multigrid finite element
method for solving Gross-Pitaevskii equation (GPE) which is a time
independent Schr\"{o}dinger equation.
The finite element method for GPE problem and general semilinear eigenvalue
problem has been investigated by \cite{Maday2010,Zhou2004}. The corresponding
error estimates are also given.

Recently, a type of multilevel correction method is proposed to solve
eigenvalue problems in \cite{LinXie_MultiLevel,Xie_JCP,Xie_IMA}.
In this multilevel correction scheme, the solution of eigenvalue problem
on the final level mesh can be reduced to a series of solutions of boundary
value problems on the multilevel meshes and a series of solutions of the
eigenvalue problem in a very low-dimensional space. Therefore, the cost of computation work
can be reduced largely. A multigrid method for the GPE has been proposed
 in \cite{XieXie2014} where a superapproximate property is founded.
Therefore, the aim of this paper is to construct a cascadic multigrid method
to solve the ground state solution of Bose-Eienstein condensates (BEC). The
cascadic multigrid method for second order elliptic eigenvalue problem
is given in \cite{HanXie2014Cascadic}.
This method transforms the eigenvalue problem solving to a series
of smoothing iteration steps on the sequence of meshes and eigenvalue
problem solving on the coarsest mesh by the multilevel correction method.
Similarly to the cascadic multigrid for the boundary value problem
\cite{BornemannDeuflhard1996,Shaidurov1996CCG}, we only do
the smoothing steps for the involved boundary value problems by using the previous
eigenpair approximation as the start value and the numbers of smoothing iteration
steps need to be increased in the coarse levels.
The order of the algebraic error for the final eigenpair approximation can arrive the same
as the discretization error of the finite element method by organizing suitable
numbers of smoothing iteration steps in different levels. The nonlinear eigenvalue
problems on a very low-dimensional space are solved by self-consistent iteration or
Newton type iteration which reduces the nonlinear eigenvalue problem to a series of linear ones.

The rest of this paper is organized as follows. In the next section, we introduce
the finite element method for the ground state solution of BEC.
A cascadic multigrid method for solving the non-dimensionalized GPE is presented
and analyzed in Section 3. In Section 4, some numerical tests are presented to
validate our theoretical analysis. Some concluding remarks are given in the last section.

\section{Finite element method for GPE problem}\label{sec;preliminary}
This section is devoted to introducing some notation and the finite element
method for the GPE problem. In this paper, we shall use the standard notation
for Sobolev spaces $W^{s,p}(\Omega)$ and their
associated norms and semi-norms (cf. \cite{Adams1975}). For $p=2$, we denote
$H^s(\Omega)=W^{s,2}(\Omega)$ and $H_0^1(\Omega)=\{v\in H^1(\Omega):\ v|_{\partial\Omega}=0\}$,
where $v|_{\Omega}=0$ is in the sense of trace, $\|\cdot\|_{s,\Omega}=\|\cdot\|_{s,2,\Omega}$.
The letter $C$ (with or without subscripts) denotes a generic
positive constant which may be different at its different occurrences through the paper.

For simplicity, we consider the following non-dimensionalized GPE problem:
Find $(\lambda, u)$ such that
\begin{equation}\label{GPE_Problem}
\left\{
\begin{array}{rcl}
-\Delta u + Wu + \zeta |u|^2u &=&\lambda u, \quad {\rm in} \  \Omega,\\
u&=&0, \ \ \quad {\rm on}\  \partial\Omega, \\
\int_{\Omega}|u|^2 d\Omega&=& 1,
\end{array}
\right.
\end{equation}
where $\Omega\subset\mathcal{R}^d$ $(d=2,3)$ is a bounded domain with
Lipschitz boundary $\partial\Omega$, $\zeta$ is some positive constant and
 $W(x)=\gamma_1 x_1^2 + \cdots + \gamma_d x_d^2\ge 0$ with $\gamma_1,\cdots,\gamma_d>0$.

In order to use the finite element method to solve
the eigenvalue problem (\ref{GPE_Problem}), we need to define
the corresponding variational form as follows:
Find $(\lambda, u )\in \mathcal{R}\times V$ such that $b(u,u)=1$ and
\begin{eqnarray}\label{Weak_GPE_Problem}
a(u,v)&=&\lambda b(u,v),\quad \forall v\in V,
\end{eqnarray}
where $V:=H_0^1(\Omega)$ and
\begin{equation}
a(u,v):=\int_{\Omega}\big(\nabla u\nabla v + Wuv + \zeta |u|^2uv\big)d\Omega,
 \ \ b(u,v):= \int_{\Omega}uv d\Omega.
\end{equation}
The existence, uniqueness and simplicity of the smallest eigenpair
of eigenvalue problem (\ref{Weak_GPE_Problem})
have been given in \cite{Maday2010}.

To simplify the notation, we also define $H^1(\Omega)$ inner-product $\widehat{a}(\cdot,\cdot)$ as
\begin{equation}\label{a_hat_def}
  \widehat{a}(w,v):=\int_{\Omega}\nabla w\nabla v d\Omega, \ \ \forall w,v\in V.
\end{equation}


Now, let us define the finite element approximations of the problem
(\ref{Weak_GPE_Problem}). First we generate a shape-regular
decomposition of the computing domain $\Omega\subset \mathcal{R}^d\
(d=2,3)$ into triangles or rectangles for $d=2$ (tetrahedrons or
hexahedrons for $d=3$). The diameter of a cell $K\in\mathcal{T}_h$
is denoted by $h_K$ and the mesh size $h$ describes  the maximum diameter of all cells
$K\in\mathcal{T}_h$.
Based on the mesh $\mathcal{T}_h$, we can construct a finite element space denoted by
 $V_h \subset V$. For simplicity, we set $V_h$ as the linear finite
 element space which is defined as follows
\begin{equation}\label{linear_fe_space}
  V_h = \big\{ v_h \in C(\Omega)\ \big|\ v_h|_{K} \in \mathcal{P}_1,
  \ \ \forall K \in \mathcal{T}_h\big\},
\end{equation}
where $\mathcal{P}_1$ denotes the linear function space.

The standard finite element scheme for eigenvalue problem (\ref{Weak_GPE_Problem}) is:
Find $(\bar{\lambda}_h, \bar{u}_h)\in \mathcal{R}\times V_h$
such that $b(\bar{u}_h,\bar{u}_h)=1$ and
\begin{eqnarray}\label{Weak_GPE_Discrete}
a(\bar{u}_h,v_h)&=&\bar{\lambda}_h b(\bar{u}_h,v_h),\quad\ \  \ \forall v_h\in V_h.
\end{eqnarray}
Then we define
\begin{equation}\label{delta_h_u_def}
  \delta_h(u):=\inf_{v_h\in V_h}\|u-v_h\|_1.
\end{equation}

\begin{lemma}(\cite[Theorem 1]{Maday2010})\label{Err_Ground_Eigen_Lem}
There exists $h_0>0$, such that for all $0<h<h_0$, the smallest eigenpair
approximation $(\bar{\lambda}_h,\bar{u}_h)$ of (\ref{Weak_GPE_Discrete})
having the following error estimates
\begin{eqnarray}
\|u-\bar{u}_h\|_1 &\le& C\delta_h(u), \\
\|u-\bar{u}_h\|_0 &\le& C\eta_a(V_h)\|u-\bar{u}_h\|_1\le C\eta_a(V_h)\delta_h(u), \\
|\lambda - \bar{\lambda}_h| &\le& C\big(\|u-\bar{u}_h\|_1^2
+ \|u-\bar{u}_h\|_0 \big) \le C\eta_a(V_h)\delta_h(u),
\end{eqnarray}
where $\eta_a(V_h)$ is defined as follows
\begin{equation}\label{eta_a_V_h_def}
\eta_a(V_h) = \sup_{f\in L^2(\Omega),\|f\|_0=1}\inf_{v_h\in V_h}\|Tf-v_h\|_1
\end{equation}
with the operator $T$ being defined as follows:
Find $Tf\in u^{\bot}$ such that
\begin{equation*}
a(Tf,v)+ 2\big(\zeta|u|^2(Tf),v\big)-\big(\lambda(Tf),v\big)=(f,v),\ \ \forall v\in u^{\bot},
\end{equation*}
where $u^{\bot}=\{v\in H_0^1(\Omega):|\int_{\Omega}uv d\Omega=0\}$.
Here we use the fact that $\delta_h(u) \le C\eta_a(V_h)$.
\end{lemma}

\section{Cascadic multigrid method for GPE}
Recently, a multilevel correction scheme is introduced in
\cite{LinXie_MultiLevel,Xie_JCP,Xie_IMA} for solving Laplace eigenvalue problems.
Based on their involved idea, we propose a type of cascadic multigrid method
for GPE problem (\ref{Weak_GPE_Problem}) in this paper.
 The main idea in this method is to approximate the underlying boundary value
problems on each level by some simple smoothing iteration steps.
In order to describe the cascadic multigrid method, we first introduce
the sequence of finite element
spaces and the properties of the concerned smoothers.

In order to design multigrid scheme, we first generate a coarse mesh $\mathcal{T}_H$
with the mesh size $H$ and the coarse linear finite element space $V_H$ is
defined on it. Then we define a
sequence of triangulations $\mathcal{T}_{h_k}$
of $\Omega\subset \mathcal{R}^d$ determined as follows.
Suppose $\mathcal{T}_{h_1}$ (produced from $\mathcal{T}_H$ by
regular refinements) is given and let $\mathcal{T}_{h_k}$ be obtained
from $\mathcal{T}_{h_{k-1}}$ via one regular refinement step
(produce $\beta^d$ subelements) such that
\begin{eqnarray}\label{mesh_size_recur}
h_k\approx\frac{1}{\beta}h_{k-1},
\end{eqnarray}
where the positive number $\beta$ denotes the refinement index and larger
 than $1$ (always equals $2$).
Based on this sequence of meshes, we construct the corresponding nested
 linear finite element spaces such that
\begin{eqnarray}\label{FEM_Space_Series}
V_{H}\subseteq V_{h_1}\subset V_{h_2}\subset\cdots\subset V_{h_n}.
\end{eqnarray}
The sequence of finite element spaces $V_{h_1}\subset V_{h_2}\subset\cdots\subset V_{h_n}$
 and the finite element space $V_H$ have  the following relations
of approximation accuracy
\begin{eqnarray}\label{delta_recur_relation}
\eta_a(V_H)\gtrsim\delta_{h_1}(u),\ \ \ \
\delta_{h_k}(u)\approx\frac{1}{\beta}\delta_{h_{k-1}}(u),\ \ \ k=2,\cdots,n.
\end{eqnarray}
In fact, since the ground eigenvalue $\lambda$ of (\ref{Weak_GPE_Problem}) is
simple (see \cite{Maday2010}) and the computing domain is convex,
 we have the following estimates
\begin{eqnarray}\label{Estimates_Eta_Delta_Hh}
\eta_a(V_H)\approx H,\ \ \ \ \eta_a(V_{h_k})\approx h_k
\ \ \ \  {\rm and}\ \ \ \ \delta_{h_k}(u)\approx h_k,\ \ \ \ k=1,\cdots,n.
\end{eqnarray}

\begin{remark}
The relation (\ref{delta_recur_relation}) is reasonable since we can choose
$\delta_{h_k}(u)=h_k\ (k=1,\cdots,n)$. Always the upper bound of
the estimate $\delta_{h_k}(u)\lesssim h_k$ holds. Recently, we also obtain the
lower bound result $\delta_{h_k}(u)\gtrsim h_k$ (c.f. \cite{LinXieXu2014}).
\end{remark}
For generality, we introduce a smoothing operator $S_h: V_h\rightarrow V_h$
which satisfies the following estimates 
\begin{equation}\label{smoothing_property}
\left\{
\begin{array}{rcl}
\|S_h^m w_h\|_1 &\le& \frac{C}{m^{\alpha}}\frac{1}{h}\|w_h\|_0, \\
\|S_h^m w_h\|_1 &\le& \|w_h\|_1, \\
\|S_h^m (w_h + v_h)\|_1 &\le&\|S_h^m w_h\|_1 + \|S_h^m v_h\|_1,
\end{array}
\right.
\end{equation}
where $C$ is a constant independent of $h$ and $\alpha$
 is some positive number depending on the choice of smoother.
It is proved in \cite{Hackbusch1985,shaidurov1995multigrid,wang2004basic}
 that the symmetric Gauss-Seidel,
the SSOR, the damped Jacobi and the Richardson iteration are smoothers in the sense
 of (\ref{smoothing_property}) with parameter $\alpha=1/2$ and the
 conjugate-gradient iteration is the
 smoother with $\alpha=1$ (cf. \cite{Shaidurov1996CCG,ShaidurovTobiska2000}).

Then we define the following notation
\begin{equation}\label{smooth_process_Vh}
w_h = {\it Smooth}(V_h, f, \xi_h, m, S_h)
\end{equation}
as the smoothing process for the following boundary value problem
\begin{equation}\label{fe_source}
\widehat{a}(u_h,v_h) = b(f,v_h), \ \ \ \ \forall v_h \in V_h,
\end{equation}
where $\xi_h$ denote the initial value of the smoothing process, $S_h$ denote the chosen
smoothing operator, $m$ the number of the iteration steps and $w_h$ is
the output of the smoothing process.

Now, we come to introduce the cascadic multigrid method for the eigenvalue problem
(\ref{Weak_GPE_Problem}).
Assume we have obtained an eigenpair approximations
 $(\lambda^{h_k}, u^{h_k}) \in \mathcal{R}\times V_{h_k}$.
We design the following cascadic type one correction step to improve the accuracy of the
current eigenpair approximation  $(\lambda^{h_k}, u^{h_k}) \in \mathcal{R}\times V_{h_k}$.
\begin{algorithm}\label{Smooth_Correction_Alg}
Cascadic type of One Correction Step
\begin{enumerate}
\item Define the following auxiliary source problem:
Find $\widehat{u}^{h_{k+1}}\in V_{h_{k+1}}$ such that
\begin{eqnarray}\label{correct_source_exact_cas}
&&\widehat{a}(\widehat{u}^{h_{k+1}},v_{h_{k+1}}) = \lambda^{h_{k}}b(u^{h_{k}},v_{h_{k+1}})\nonumber\\
 && \quad\quad\quad - \big((W+\zeta|u^{h_k}|^2)u^{h_k},v_{h_{k+1}}\big),
\ \ \forall v_{h_{k+1}}\in V_{h_{k+1}}.
\end{eqnarray}
Perform the smoothing process (\ref{smooth_process_Vh}) to obtain a new eigenfuction
 approximation $\widetilde{u}^{h_{k+1}}\in V_{h_{k+1}}$ by
\begin{equation}\label{smooth_correct_process}
\widetilde{u}^{h_{k+1}} = {\it Smooth}(V_{h_{k+1}},
\lambda^{h_{k}}u^{h_{k}}-(W+\zeta|u^{h_k}|^2)u^{h_k},
 u^{h_{k}}, m_{k+1}, S_{h_{k+1}}).
\end{equation}

\item Define a new finite element space $V_{H}^{h_{k+1}} = V_H +
{\rm span}\{\widetilde{u}^{h_{k+1}}\}$ and solve the following eigenvalue problem:
Find $(\lambda^{h_{k+1}},u^{h_{k+1}})\in \mathcal{R}\times V_{H}^{h_{k+1}}$
such that\\ $b(u^{h_{k+1}},u^{h_{k+1}})=1$ and
\begin{equation}\label{cascadic_correct_eig_exact}
a(u^{h_{k+1}},v_{H}^{h_{k+1}}) = \lambda^{h_{k+1}}b(u^{h_{k+1}},v_{H}^{h_{k+1}}),
\ \ \ \ \ \forall v_{H}^{h_{k+1}}\in V_{H}^{h_{k+1}}.
\end{equation}
\end{enumerate}
Summarize the above two steps by defining
\begin{eqnarray*}
(\lambda^{h_{k+1}},u^{h_{k+1}}) =
{\it SmoothCorrection}(V_H,V_{h_{k+1}},\lambda^{h_{k}},u^{h_{k}},m_{k+1},S_{h_{k+1}}).
 \end{eqnarray*}
\end{algorithm}

Based on the above algorithm, i.e., the cascadic type of one correction step, we can
construct a cascadic multigrid method for GPE as follows:

\begin{algorithm}\label{Cascadic_MCS_Alg}
GPE Cascadic Multigrid Method
\begin{enumerate}
\item Solve the following GPE problem in the initial finite element space $V_{h_1}$:
Find $(\lambda^{h_1}, u^{h_1})\in \mathcal{R}\times V_{h_1}$ such that
\begin{equation*}
a(u^{h_1}, v_{h_1}) = \lambda^{h_1} b(u^{h_1}, v_{h_1}), \quad \forall v_{h_1}\in  V_{h_1}.
\end{equation*}
\item For $k=1,\cdots,n-1$, do the following iteration
\begin{eqnarray*}
(\lambda^{h_{k+1}},u^{h_{k+1}}) =
{\it SmoothCorrection}(V_H,V_{h_{k+1}},\lambda^{h_{k}},u^{h_{k}},m_{k+1},S_{h_{k+1}}).
\end{eqnarray*}
End Do
\end{enumerate}
Finally, we obtain an eigenpair approximation
$(\lambda^{h_n},u^{h_n}) \in \mathcal{R}\times V_{h_n}$.
\end{algorithm}

\vspace{1ex} 
In order to analyze the convergence of Algorithm \ref{Cascadic_MCS_Alg},
 we introduce an auxiliary algorithm and then show its superapproximate property.
Similarly, assume we have obtained an eigenpair approximations
 $(\widetilde{\lambda}_{h_k}, \widetilde{u}_{h_k}) \in \mathcal{R}\times V_{h_k}$.
The following auxiliary one correction step is defined as follows.
\begin{algorithm}\label{Auxiliary_Correction_Alg}
Auxiliary One Correction Step
\begin{enumerate}
\item Solve the following auxiliary source problem:
Find $\widehat{u}_{h_{k+1}}\in V_{h_{k+1}}$ such that
\begin{eqnarray}\label{Auxiliary_correct_source_exact}
&&\widehat{a}(\widehat{u}_{h_{k+1}},v_{h_{k+1}})
= \widetilde{\lambda}_{h_k}b(\widetilde{u}_{h_k},v_{h_{k+1}})\nonumber\\
&&\quad\quad\quad
-\big(\big(W+\zeta|\widetilde{u}_{h_k}|^2\big)\widetilde{u}_{h_k},v_{h_{k+1}}\big),
\ \ \forall v_{h_{k+1}}\in V_{h_{k+1}}.
\end{eqnarray}
\item Define a new finite element space
$\widetilde{V}_{H,h_{k+1}} = V_H + {\rm span}\{\widehat{u}_{h_{k+1}}\}
 + {\rm span}\{\widetilde{u}^{h_{k+1}}\}$
 and solve the following eigenvalue problem:
Find $(\widetilde{\lambda}_{h_{k+1}},\widetilde{u}_{h_{k+1}})\in
 \mathcal{R}\times \widetilde{V}_{H,h_{k+1}}$ such that
  $b(\widetilde{u}_{h_{k+1}},\widetilde{u}_{h_{k+1}})=1$ and
\begin{equation}\label{Auxiliary_correct_eig_exact}
a(\widetilde{u}_{h_{k+1}},\widetilde{v}_{H,h_{k+1}}) = \widetilde{\lambda}_{h_{k+1}}b(\widetilde{u}_{h_{k+1}},\widetilde{v}_{H,h_{k+1}}),
\ \ \ \ \ \forall \widetilde{v}_{H,h_{k+1}}\in \widetilde{V}_{H,h_{k+1}}.
\end{equation}
\end{enumerate}
Summarize the above two steps by defining
\begin{eqnarray*}
(\widetilde{\lambda}_{h_{k+1}},\widetilde{u}_{h_{k+1}}) =
{\it AuxiliaryCorrection}(V_H,V_{h_{k+1}},\widetilde{\lambda}_{h_{k}},\widetilde{u}_{h_{k}},
\widetilde{u}^{h_{k+1}}).
 \end{eqnarray*}
\end{algorithm}

\begin{algorithm}\label{Auxiliary_MCS_Alg}
GPE Auxiliary Multilevel Correction Method
\begin{enumerate}
\item Solve the following GPE problem in the initial finite element space $V_{h_1}$:
Find $(\widetilde{\lambda}_{h_1}, \widetilde{u}_{h_1})\in \mathcal{R}\times V_{h_1}$ such that
\begin{equation*}
a(\widetilde{u}_{h_1}, v_{h_1}) = \widetilde{\lambda}_{h_1} b(\widetilde{u}_{h_1}, v_{h_1}),
\quad \forall v_{h_1}\in  V_{h_1}.
\end{equation*}
\item For $k=1,\cdots,n-1$, do the following iteration
\begin{eqnarray*}
(\widetilde{\lambda}_{h_{k+1}},\widetilde{u}_{h_{k+1}}) =
{\it AuxiliaryCorrection}(V_H,V_{h_{k+1}},\widetilde{\lambda}_{h_{k}},\widetilde{u}_{h_{k}},
\widetilde{u}^{h_{k+1}}).
 \end{eqnarray*}
 End Do
 \end{enumerate}
  Finally, we obtain an eigenpair approximation
$(\widetilde{\lambda}_{h_{n}},\widetilde{u}_{h_{n}})
\in \mathcal{R}\times V_{h_{n}}$.
\end{algorithm}
\vspace{1ex}
Before analyzing the convergence of Algorithm \ref{Cascadic_MCS_Alg},
we show a superapproximate
property of $\widetilde{u}_{h_k}$ obtained by Algorithm \ref{Auxiliary_MCS_Alg}. The similar
result is also analyzed in \cite{XieXie2014}.
\begin{theorem}\label{super_approx_thm}
Assume $\widetilde{u}_{h_k}$ ($k=1,\cdots,n$) are obtained by Algorithm \ref{Auxiliary_MCS_Alg}
and $\bar{u}_{h_k}$ ($k=1,\cdots,n$) the standard finite
element solution in $V_{h_k}$.  If the sequence of finite
element spaces $V_{h_1}, \cdots,
V_{h_{n}}$ and the coarse finite element space $V_H$
satisfy the following condition
\begin{equation}\label{fe_sequence_condition}
C\eta_a(V_H)\beta^2 < 1,
\end{equation}
the following estimate holds
\begin{equation}\label{super_approx_eigenfunction}
\|\bar{u}_{h_k} - \widetilde{u}_{h_k}\|_1 \le C\eta_a(V_{h_k})
\delta_{h_k}(u),\ \ \ \ \ \ \ k=1,\cdots,n,
\end{equation}
and
\begin{equation}\label{super_approx_eigenfunction_b_norm}
  \|\bar{u}_{h_k} - \widetilde{u}_{h_k}\|_0 \le C\eta_a(V_{h_k})
  \delta_{h_k}(u),\ \ \ \ \ \ \ k=1,\cdots,n,
\end{equation}
where $C$ is a constant only depending on the eigenvalue $\lambda$.
The eigenvalue approximations $\widetilde{\lambda}_{h_k}$ and
$\bar{\lambda}_{h_k}$ have the following estimates
\begin{equation}\label{super_approx_eigenvalue}
\big| \bar{\lambda}_{h_k} - \widetilde{\lambda}_{h_k} \big|
\le C\eta_a(V_{h_k})\delta_{h_k}(u),\ \ \ \ \ \ \ k=1,2,\cdots,n.
\end{equation}
\end{theorem}
\begin{proof}
Define $\varepsilon_{h_k}:=|\widetilde{\lambda}_{h_k}-\bar{\lambda}_{h_k}|+
\|\widetilde{u}_{h_k}-\bar{u}_{h_k}\|_0,\ k=1,2,\cdots,n$. And it is obvious that
$\varepsilon_{h_1}=0$.
From (\ref{Weak_GPE_Discrete}) and (\ref{Auxiliary_correct_source_exact}), we have
\begin{eqnarray*}
&&  \widehat{a}(\bar{u}_{h_{k+1}} - \widehat{u}_{h_{k+1}},v_{h_{k+1}}) \nonumber\\
&=&b(\bar{\lambda}_{h_{k+1}}\bar{u}_{h_{k+1}}-
\widetilde{\lambda}_{h_k}\widetilde{u}_{h_k},v_{h_{k+1}}) \\
&&+ \big(\big(W+\zeta|u^{h_k}|^2\big)u^{h_k}
-\big(W+\zeta|\bar{u}_{h_{k+1}}|^2\big)\bar{u}_{h_{k+1}},
v_{h_{k+1}}\big) \nonumber\\
&\le& C\big(\|\bar{\lambda}_{h_{k+1}}\bar{u}_{h_{k+1}}-\widetilde{\lambda}_{h_k}
\widetilde{u}_{h_k}\|_0 \|v_{h_{k+1}}\|_1  \\
&&\quad\quad + \|\bar{u}_{h_{k+1}}-\widetilde{u}_{h_k}\|_0
\big(\|\bar{u}_{h_{k+1}}\|_{0,6,\Omega}^2
+\|\widetilde{u}_{h_k}\|_{0,6,\Omega}^2\big)\|v_{h_{k+1}}\|_1 \big) \\
&\le& C\big(\|\bar{\lambda}_{h_{k+1}}\bar{u}_{h_{k+1}}-\bar{\lambda}_{h_k}\bar{u}_{h_k}\|_0
+\|\bar{\lambda}_{h_k}\bar{u}_{h_k}-\widetilde{\lambda}_{h_k}\widetilde{u}_{h_k}\|_0  \\
&&+ \big(\|\bar{u}_{h_{k+1}}-\bar{u}_{h_k}\|_0+ \|\bar{u}_{h_k}-\widetilde{u}_{h_k}\|_0\big)
\big(\|\bar{u}_{h_{k+1}}\|_1^2+\|\widetilde{u}_{h_k}\|_1^2\big)\big)\|v_{h_{k+1}}\|_1 \\
&\le& C\big(|\bar{\lambda}_{h_{k+1}}-\bar{\lambda}_{h_k}|
+\|\bar{u}_{h_{k+1}}-\bar{u}_{h_k}\|_0+\varepsilon_{h_k}\big)\|v_{h_{k+1}}\|_1.
\end{eqnarray*}
It leads to the following estimates
\begin{eqnarray}\label{Error_1}
\|\bar{u}_{h_{k+1}} - \widehat{u}_{h_{k+1}}\|_1 \le
C\big(|\bar{\lambda}_{h_{k+1}}-\bar{\lambda}_{h_k}|
+\|\bar{u}_{h_{k+1}}-\bar{u}_{h_k}\|_0+\varepsilon_{h_k} \big).
\end{eqnarray}
Note that the eigenvalue problem (\ref{Auxiliary_correct_eig_exact}) can be regarded
 as a finite dimensional subspace approximation of the eigenvalue problem (\ref{Weak_GPE_Discrete}).
Similarly to Lemma \ref{Err_Ground_Eigen_Lem} (see \cite{Maday2010}),
from the second step in Algorithm \ref{Auxiliary_Correction_Alg},
the following estimate holds
\begin{equation}\label{Error_2}
\|\bar{u}_{h_{k+1}} - \widetilde{u}_{h_{k+1}} \|_1\le
C\inf_{\widetilde{v}_{H,h_{k+1}}\in\widetilde{V}_{H,h_{k+1}}}
\|\bar{u}_{h_{k+1}}-\widetilde{v}_{H,h_{k+1}}\|_1 \le
C \| \bar{u}_{h_{k+1}} - \widehat{u}_{h_{k+1}} \|_1.
\end{equation}
Then combining (\ref{Error_1}) and (\ref{Error_2}) leads to
\begin{eqnarray}\label{super_approx_recur_break_1}
\| \bar{u}_{h_{k+1}} - \widetilde{u}_{h_{k+1}} \|_1
\le C\big(|\bar{\lambda}_{h_{k+1}}-\bar{\lambda}_{h_k}|
+\|\bar{u}_{h_{k+1}}-\bar{u}_{h_k}\|_0+\varepsilon_{h_k}\big).
\end{eqnarray}
From the properties of $V_{h_k}\subset V_{h_{k+1}}$, $\widetilde{V}_{H,h_k}\subset V_{h_k}$,
Lemma \ref{Err_Ground_Eigen_Lem} and (\ref{delta_recur_relation}), we have
\begin{eqnarray*}
&&\|\bar{u}_{h_{k+1}} - \bar{u}_{h_{k}}\|_1\le C\delta_{h_{k}}(u), \ \ \
\|\bar{u}_{h_{k+1}} - \bar{u}_{h_{k}}\|_0
\le C\eta_a(V_{h_k})\|\bar{u}_{h_{k+1}} - \bar{u}_{h_{k}}\|_1,\nonumber\\
&&\big| \bar{\lambda}_{h_{k+}} - \bar{\lambda}_{h_{k}} \big|
\leq C\eta_a(V_{h_k})\delta_{h_k}(u),\ \ \
\| \bar{u}_{h_k} - \widetilde{u}_{h_k} \|_0 \le C\eta_a(V_H)
\| \bar{u}_{h_k} - \widetilde{u}_{h_k} \|_1,\\
&&\big| \bar{\lambda}_{h_{k}} - \widetilde{\lambda}_{h_{k}} \big|
\le C\eta_a(V_H)\| \bar{u}_{h_k} - \widetilde{u}_{h_k} \|_1.
\end{eqnarray*}
Substituting above inequalities into (\ref{super_approx_recur_break_1})
leads to the following estimates
\begin{eqnarray}\label{super_approx_recur_break_2}
\| \bar{u}_{h_{k+1}} - \widetilde{u}_{h_{k+1}} \|_1 &\le& C\big(
\eta_a(V_{h_k})\delta_{h_k}(u) + \varepsilon_{h_k} \big) \nonumber\\
&\le& C\big(\eta_a(V_{h_k})\delta_{h_k}(u) + \eta_a(V_H)\|\bar{u}_{h_k} - \widetilde{u}_{h_k}\|_1 \big).
\end{eqnarray}
When $k=1$, since $\widetilde{u}_{h_1}:=\bar{u}_{h_1}$ and
$\widetilde{\lambda}_{h_1}:=\bar{\lambda}_{h_1}$, we have
\begin{eqnarray}\label{Error_k_1}
\| \bar{u}_{h_{2}} - \widetilde{u}_{h_{2}} \|_1 \leq C\eta_a(V_{h_1})\delta_{h_1}(u).
\end{eqnarray}
Based on (\ref{delta_recur_relation}), (\ref{super_approx_recur_break_2}), (\ref{Error_k_1})
and recursive argument, we have the following estimates:
\begin{eqnarray}\label{super_approx_recur_break_3}
\|\bar{u}_{h_{k}} - \widetilde{u}_{h_{k}}\|_1 &\le& C\sum_{j=2}^k \big(C\eta_a(V_H)\big)^{k-j}
\eta_a(V_{h_{j-1}})\delta_{h_{j-1}}(u)  \nonumber \\
&\le& C\sum_{j=2}^k \big(C\eta_a(V_H)\big)^{k-j}
\beta^{k-j+1}\eta_a(V_{h_k})\beta^{k-j+1}\delta_{h_k}(u) \nonumber\\
&\le& C{\beta}^2\Big(\sum_{j=2}^k \big(C\eta_a(V_H)\beta^2\big)^{k-j}\Big)
\eta_a(V_{h_k})\delta_{h_k}(u)\nonumber\\
&\le& \frac{C\beta^2}{1-C\beta^2\eta_a(V_H)}\eta_a(V_{h_k})\delta_{h_k}(u).
\end{eqnarray}
Therefore,  the desired result (\ref{super_approx_eigenfunction}) holds under
the condition $C\eta_a(V_H)\beta^2<1$. Furthermore, (\ref{super_approx_eigenfunction_b_norm})
and (\ref{super_approx_eigenvalue}) can be obtained directly from Lemma \ref{Err_Ground_Eigen_Lem}
and the property $\widetilde{V}_{H,h_{k+1}}\subset V_{h_{k+1}}$.
\end{proof}

\vspace{1ex}
Note that $V_H^{h_k}\subset\widetilde{V}_{H,h_k}$,
 we can obtain the following estimates which play an important role in our analysis.

\begin{lemma}\cite[Theorem 1]{Maday2010}\label{E_h_less_P_h_Lem}
Let $u^{h_k}$, $V_H^{h_k}$ and $\widetilde{u}_{h_k}$, $\widetilde{V}_{H,h_k}$ be
defined in Algorithms \ref{Smooth_Correction_Alg} and \ref{Auxiliary_Correction_Alg}.
Then the following estimates hold:
\begin{eqnarray}
\|u^{h_k} - \widetilde{u}_{h_k}\|_1 &\le&
 C\|\widehat{u}_{h_k}-\widetilde{u}^{h_k}\|_1, \label{E_h_less_P_h_a}\\
\|u^{h_k} - \widetilde{u}_{h_k}\|_0 &\le& C\eta_a(V_H)\|u^{h_k}
- \widetilde{u}_{h_k}\|_1, \label{E_h_less_P_h_b}\\
|\lambda^{h_k} - \widetilde{\lambda}_{h_k}| &\le&
C \eta_a(V_H)\|u^{h_k} - \widetilde{u}_{h_k}\|_1. \label{E_h_less_P_h_eigenvalue}
\end{eqnarray}
\end{lemma}
\begin{proof}
Since $V_H^{h_k}\subset\widetilde{V}_{H,h_k}$, according to (\ref{cascadic_correct_eig_exact})
 and (\ref{Auxiliary_correct_eig_exact}), $u^{h_k}$
  can be viewed as the spectral projection of $\widetilde{u}_{h_k}$.
  Then from Lemma \ref{Err_Ground_Eigen_Lem} and the definitions of $\widetilde{V}_{H,h_k}$
  and $V_H^{h_k}$, we have
\begin{eqnarray}
\|\widetilde{u}_{h_k}-u^{h_k}\|_1 &\le& C\inf_{v_H^{h_k}
\in V_H^{h_k}}\|\widetilde{u}_{h_k}-v_H^{h_k}\|_1 \le
 C\inf_{v_H^{h_k}\in V_H^{h_k}}\|\widehat{u}_{h_k}-v_H^{h_k}\|_1 \nonumber\\
&\le& C\|\widehat{u}_{h_k}-\widetilde{u}^{h_k}\|_1,
\end{eqnarray}
which is the desired result (\ref{E_h_less_P_h_a}).

Similarly, we also have (\ref{E_h_less_P_h_b}) by the following argument
\begin{eqnarray*}
\|\widetilde{u}_{h_{k}} - u^{h_{k}}\|_0 \le
 C\eta_a(V_H^{h_k})\|\widetilde{u}_{h_{k}} - u^{h_{k}}\|_1
\le C\eta_a(V_H)\|\widetilde{u}_{h_{k}} - u^{h_{k}}\|_1.
\end{eqnarray*}
Furthermore, (\ref{E_h_less_P_h_eigenvalue}) can be obtained directly
from Lemma \ref{Err_Ground_Eigen_Lem}
and the proof is complete.
\end{proof}

\begin{remark}
Since $V_H\subset V_H^{h_k}$ and $V_H\subset \widetilde{V}_{H,h_k}$,
from Lemma \ref{Err_Ground_Eigen_Lem}, we have
\begin{equation}\label{u_hat_u_tilde_a_less_H}
\|u^{h_{k}} - \widetilde{u}_{h_{k}}\|_1 \le \|u^{h_{k}} - u\|_1 +
\| u- \widetilde{u}_{h_{k}}\|_1 \le C\delta_{H}(u).
\end{equation}
\end{remark}

\vspace{1ex} 
Now, we come to give error estimates for Algorithm \ref{Cascadic_MCS_Alg}.
\begin{theorem}\label{Cascadic_Convergence_Thm}
Assume the eigenpair approximation $(\lambda^{h_n}, u^{h_n})$ is
obtained by Algorithm \ref{Cascadic_MCS_Alg},
$(\widetilde{\lambda}_{h_n}, \widetilde{u}_{h_n})$
is obtained by Algorithm \ref{Auxiliary_MCS_Alg} and the smoother selected in each
level $V_{h_k}$ satisfy the smoothing property (\ref{smoothing_property})
 for $k=1, \cdots, n$.
Under the conditions of Theorem \ref{super_approx_thm},
we have the following estimate
\begin{equation}\label{cascadic_convergence_eigenfunction}
\|\widetilde{u}_{h_n} - u^{h_n}\|_1 \le
C\sum_{k=2}^n \frac{\big(1+C\eta_a(V_H)\big)^{n-k}}{m_k^{\alpha}}\delta_{h_k}(u),
\end{equation}
and the corresponding eigenvalue error estimate
\begin{equation}\label{cascadic_convergence_eigenvalue}
\big| \widetilde{\lambda}_{h_n} - \lambda^{h_n} \big|
\le C\eta_a(V_H)\sum_{k=2}^n \frac{\big(1+C\eta_a(V_H)\big)^{n-k}}{m_k^{\alpha}}\delta_{h_k}(u).
\end{equation}
\end{theorem}

\begin{proof}
 Define $e_{h_k}:=u^{h_k} - \widetilde{u}_{h_k}$ for $k=1, \cdots, n$.
 Then it is easy to see that $e_{h_1}=0$.

 From Lemma \ref{E_h_less_P_h_Lem}, the following inequalities hold
\begin{eqnarray}
\|e_{h_{k+1}}\|_1 &=& \|u^{h_{k+1}} - \widetilde{u}_{h_{k+1}}\|_1
\le C\|\widehat{u}_{h_{k+1}} - \widetilde{u}^{h_{k+1}}\|_1 \nonumber\\
&\le& C\big(\|\widehat{u}_{h_{k+1}} - \widehat{u}^{h_{k+1}}\|_1
+ \|\widehat{u}^{h_{k+1}} - \widetilde{u}^{h_{k+1}}\|_1\big).
\label{cascadic_convergence_break_1}
\end{eqnarray}
For the first term in (\ref{cascadic_convergence_break_1}), together
with (\ref{correct_source_exact_cas}), (\ref{Auxiliary_correct_source_exact}), Lemma \ref{E_h_less_P_h_Lem}
 and (\ref{u_hat_u_tilde_a_less_H}), we have
\begin{eqnarray}\label{cascadic_convergence_break_2}
\|\widehat{u}_{h_{k+1}} - \widehat{u}^{h_{k+1}}\|_1 &\le& C\|\lambda^{h_k}u^{h_k}
- \widetilde{\lambda}_{h_k}\widetilde{u}_{h_k}\|_0\nonumber\\
&\le& C\big(|\lambda^{h_k}-\widetilde{\lambda}_{h_k}|
+ \|u^{h_k}-\widetilde{u}_{h_k}\|_0 \big)\nonumber\\
&\le& C\eta_a(V_H)\|u^{h_k}-\widetilde{u}_{h_k}\|_1\nonumber\\
&=& C\eta_a(V_H)\|e_{h_k}\|_1.
\end{eqnarray}
For the second term in (\ref{cascadic_convergence_break_1}),
 due to (\ref{smoothing_property}) and (\ref{cascadic_convergence_break_2}),
 the following estimates hold
\begin{eqnarray}
&&\|\widehat{u}^{h_{k+1}} - \widetilde{u}^{h_{k+1}}\|_1  =
\|S_{h_{k+1}}^{m_{k+1}}(\widehat{u}^{h_{k+1}} - u^{h_k})\|_1   \nonumber \\
&\le& \|S_{h_{k+1}}^{m_{k+1}}(\widehat{u}^{h_{k+1}} - \widetilde{u}_{h_k})\|_1
+ \|S_{h_{k+1}}^{m_{k+1}}(\widetilde{u}_{h_k} - u^{h_k})\|_1  \nonumber \\
&\le& \|S_{h_{k+1}}^{m_{k+1}}(\widehat{u}^{h_{k+1}}-\widehat{u}_{h_{k+1}})\|_1
+ \|S_{h_{k+1}}^{m_{k+1}}(\widehat{u}_{h_{k+1}} - \widetilde{u}_{h_{k}})\|_1
+ \|\widetilde{u}_{h_k} - u^{h_k}\|_1\nonumber \\
&\le& \|\widehat{u}_{h_{k+1}} - \widehat{u}^{h_{k+1}}\|_1 +
\frac{C}{m_{k+1}^{\alpha}}\frac{1}{h_{k+1}}\|\widehat{u}_{h_{k+1}}
- \widetilde{u}_{h_{k}}\|_0 + \|\widetilde{u}_{h_k} - u^{h_k}\|_1 \nonumber \\
&\le&\big(1+C\eta_a(V_H)\big)\|e_{h_k}\|_1
+ \frac{C}{m_{k+1}^{\alpha}}\frac{1}{h_{k+1}}\|\widehat{u}_{h_{k+1}}
- \widetilde{u}_{h_{k}}\|_0. \label{cascadic_convergence_break_3}
\end{eqnarray}
According to Lemma \ref{Err_Ground_Eigen_Lem}, (\ref{delta_recur_relation}),
Theorem \ref{super_approx_thm} and its proof, we have
\begin{eqnarray} \label{cascadic_convergence_break_4}
\|\widehat{u}_{h_{k+1}} - \widetilde{u}_{h_{k}}\|_0 &\le&
\|\widehat{u}_{h_{k+1}} - \bar{u}_{h_{k+1}}\|_0 +
\|\bar{u}_{h_{k+1}} - \bar{u}_{h_{k}}\|_0
+ \|\bar{u}_{h_{k}} - \widetilde{u}_{h_{k}}\|_0 \nonumber\\
&\le& C\eta_a(V_{h_{k+1}})\delta_{h_{k+1}}(u).
\end{eqnarray}
Combining (\ref{cascadic_convergence_break_1}), (\ref{cascadic_convergence_break_2}),
(\ref{cascadic_convergence_break_3}), (\ref{cascadic_convergence_break_4})
 and (\ref{Estimates_Eta_Delta_Hh}), we have
\begin{equation}\label{e_k_recur}
\|e_{h_{k+1}}\|_1 \le \big(1+C\eta_a(V_H)\big)\|e_{h_k}\|_1
+ \frac{C}{m_{k+1}^{\alpha}}\delta_{h_{k+1}}(u), \ \ k=1,\cdots,n-1.
\end{equation}
Based on (\ref{e_k_recur}), the fact $e_{h_1}=0$ and the recursive argument,
the following estimates hold
\begin{eqnarray*}
\|e_{h_n}\|_1 &\le& \big(1+C\eta_a(V_H)\big)\|e_{h_{n-1}}\|_1
+ \frac{C}{m_{n}^{\alpha}}\delta_{h_{n}}(u)  \\
&\le& \big(1+C\eta_a(V_H)\big)^2\|e_{h_{n-2}}\|_1
+ \big(1+C\eta_a(V_H)\big)\frac{C}{m_{n-1}^{\alpha}}\delta_{h_{n-1}}(u)
+\frac{C}{m_{n}^{\alpha}}\delta_{h_{n}}(u) \\
&\le& C\sum_{k=2}^n \big(1+C\eta_a(V_H)\big)^{n-k}\frac{1}{m_k^{\alpha}}\delta_{h_k}(u).
\end{eqnarray*}
This is the desired result (\ref{cascadic_convergence_eigenfunction}).
The estimate (\ref{cascadic_convergence_eigenvalue}) can be obtained
from Lemma \ref{Err_Ground_Eigen_Lem} and (\ref{cascadic_convergence_eigenfunction}).
\end{proof}

\begin{corollary}
Under the conditions of Theorem \ref{Cascadic_Convergence_Thm}, we have the following estimates:
\begin{eqnarray}
\|\bar{u}_{h_n} - u^{h_n}\|_1 &\le& C\Big( \eta_a(V_{h_n})\delta_{h_n}(u)+  \sum_{k=2}^n \frac{\big(1+C\eta_a(V_H)\big)^{n-k}}{m_k^{\alpha}}\delta_{h_k}(u)  \Big)\label{u_bar_u_h},\\
|\bar{\lambda}_{h_n} - \lambda^{h_n}| &\le& C\Big( \eta_a(V_{h_n})\delta_{h_n}(u)+  \sum_{k=2}^n \frac{\big(1+C\eta_a(V_H)\big)^{n-k}}{m_k^{\alpha}}\delta_{h_k}(u)  \Big)\label{lambda_bar_lambda_h}.
\end{eqnarray}
\end{corollary}

\vspace{1ex}
Now we come to estimate the computational work for Algorithm \ref{Cascadic_MCS_Alg}.
 Define the dimension of each linear finite element space as
 \begin{equation*}
   N_k:=\text{\rm dim}~V_{h_k},\ \ \ \ k=1,\cdots,n.
 \end{equation*}
Then we have
\begin{equation}\label{N_h_recur_relation}
  N_k\approx\Big(\frac{h_k}{h_n}\Big)^{-d}N_n = \Big(\frac{1}{\beta}\Big)^{d(n-k)}N_n,\ \ \  \ k=1,\cdots,n.
\end{equation}
Different from the linear Laplace eigenvalue case, in the second step of
Algorithm \ref{Cascadic_MCS_Alg}, we have to solve a nonlinear
eigenvalue problem on the newly constructed coarse space
$V_H^{h_k}$. Always, some type of nonlinear iteration method is used
 to solve this nonlinear eigenvalue problem. In each nonlinear iteration step,
 we need to assemble the stiff matrix on the finite element space $V_H^{h_k}$ $(k=2,\cdots,n)$,
  which needs the computational work $\mathcal{O}(N_k)$. Fortunately,
   the matrix assembling can be carried out by the parallel way easily
   in the finite element space since it has no data transfer.

From Theorem \ref{Cascadic_Convergence_Thm}, in order to control the global error,
 it is required that the number of smoothing iterations in the coarser spaces
 should be larger than the fine spaces. To give a precise analysis for the final
 error and complexity estimates, we assume the following inequality holds for the
  number of smoothing iterations in each level mesh:
\begin{equation}\label{iter_reccur_relation}
\Big(\frac{h_k}{h_{n}}\Big)^{\zeta} \le \frac{m_k^{\alpha}}{\bar{m}^{\alpha}}
 \le \sigma \Big(\frac{h_k}{h_{n}}\Big)^{\zeta},\ \ \ \ \ \ \  k=2,\cdots,n-1,
\end{equation}
where $ \bar{m} = m_{n} $, $\sigma > 1$ and $\zeta >1$ are some appropriate constants.

Now, we give the final error and the complexity estimates for
Algorithm \ref{Cascadic_MCS_Alg}.
\begin{theorem}\label{estimate_number_iter}
Under the conditions  (\ref{delta_recur_relation}),
 (\ref{iter_reccur_relation}) and $\beta^{1-\zeta}(1+CH)<1$,
 for any given $\gamma \in (0,1]$, the final error estimate
\begin{eqnarray}\label{Final_Error_Estimate}
\|u^{h_{n}} - \widetilde{u}_{h_{n}}\|_1 \le \gamma h_{n}
\end{eqnarray}
holds if we take
\begin{eqnarray}\label{Condition_m_bar}
\bar{m}>\Big(\frac{CC_{\zeta}}{\gamma}\Big)^{\frac{1}{\alpha}},
\end{eqnarray}
where $C_\zeta = {1}/{(1-\beta^{1-\zeta}(1+CH))}$.

Assume the GPE problem solved in the coarse spaces $V_{H}$ and $V_{h_1}$ need work
$M_H$ and $M_{h_1}$, respectively. We use $P$
 computing-nodes in Algorithm \ref{Cascadic_MCS_Alg},
 and let $\varpi$ denote the nonlinear iteration times
 when we solve the nonlinear eigenvalue
 problem (\ref{cascadic_correct_eig_exact}).
If $\zeta/\alpha <d$, the total computational work of
 Algorithm \ref{Cascadic_MCS_Alg} can be bounded by
$\mathcal{O}\Big(\big(1+\frac{\varpi}{p}\big)N_{n}+M_{h_1}+M_H\log(N_{n})\Big)$
 and furthermore $\mathcal{O}(N_{n})$
provided  $M_H \ll N_{n}$, $M_{h_1}\leq N_{n}$ and $\frac{\varpi}{p}\le C$.
While if $\zeta/\alpha =d$, the total computational work can be bounded by
$\mathcal{O}((1+\varpi/p)N_{n}\log(N_{n})+M_{h_1}+M_H\log(N_{n}))$
and furthermore $\mathcal{O}(N_{n}\log(N_{n}))$
provided  $M_H \ll N_{n}$, $M_{h_1}\leq N_{n}$ and $\varpi/p\leq C$.
\end{theorem}
\begin{proof}
By Theorem \ref{Cascadic_Convergence_Thm}, together
with (\ref{mesh_size_recur}), (\ref{Estimates_Eta_Delta_Hh}),
(\ref{cascadic_convergence_eigenfunction}), (\ref{iter_reccur_relation})
and $\beta^{1-\zeta}(1+CH)<1$, we have the following estimates
\begin{eqnarray}\label{m_k_cascadic_convergence}
\|u^{h_{n}} - \widetilde{u}_{h_{n}}\|_1 &\leq& C \sum_{k=2}^{n}(1+C\eta_a(V_H))^{n-k}
\frac{1}{m_k^{\alpha}}\delta_{h_k}(u) \nonumber\\
&\le& C \sum_{k=2}^{n}(1+CH)^{n -k} \frac{1}{\bar{m}^{\alpha}}
 \Big(\frac{h_k}{h_{n}}\Big)^{-\zeta}h_k \nonumber\\
&\le& C\sum_{k=2}^{n} (1+CH)^{n -k}
\beta^{(n-k)(1-\zeta)}\frac{h_{n}}{\bar{m}^{\alpha}} \nonumber\\
&=& C\frac{h_{n}}{\bar{m}^{\alpha}}\sum_{k=0}^{n-2}
 \big(\beta^{1-\zeta}(1+CH)\big)^{k}\nonumber\\
&\leq& C\frac{h_{n}}{\bar{m}^{\alpha}}\frac{1}{1-\beta^{1-\zeta}(1+CH)}\nonumber\\
&\le& \frac{CC_{\zeta}}{\bar{m}^{\alpha}}h_{n}.
\end{eqnarray}
 Then it is obvious that we can obtain
$\|u^{h_{n}} - \widetilde{u}_{h_{n}}\|_1 \le \gamma h_{n}$ when $\bar{m}$ satisfies the condition
(\ref{Condition_m_bar}).

Let $W$ denote the whole computational work of Algorithm \ref{Cascadic_MCS_Alg},
$w_k$ the work on the $k$-th level for $k=1,\cdots,n$.
Based on the definition of
Algorithms \ref{Smooth_Correction_Alg} and \ref{Cascadic_MCS_Alg},
(\ref{mesh_size_recur}), (\ref{N_h_recur_relation}) and (\ref{iter_reccur_relation}),
 the following estimates hold
\begin{eqnarray*}
W &=& \sum_{k=1}^{n} w_k \le M_{h_1} + \sum_{k=2}^{n} m_k N_k + \sum_{k=2}^{n} \frac{\varpi}{p} N_k
+ M_H \log_{\beta}(N_{n}) \\
&\leq& M_{h_1} + C M_H\log(N_{n})+ \bar{m}\sigma^{1/\alpha}
N_{n}\sum_{k=2}^{n} \Big(\frac{1}{\beta}\Big)^{(n-k)(d-\zeta/\alpha)}\nonumber\\
&&\ \ \ \ + \frac{\varpi}{p}
N_{n}\sum_{k=2}^{n} \Big(\frac{1}{\beta}\Big)^{d(n-k)} \\
&\leq& M_{h_1} + C M_H\log(N_{n})+C \frac{\varpi}{p} N_{n}+ \bar{m}\sigma^{1/\alpha}
N_{n}\sum_{k=2}^{n} \Big(\frac{1}{\beta}\Big)^{(n-k)(d-\zeta/\alpha)}.
\end{eqnarray*}
Then we know that the computational work $W$ can be bounded by
$\mathcal{O}\big(M_{h_1}+M_H\log(N_{n})+\big(1+\varpi/p\big)N_{n}\big)$
when $d-\zeta/\alpha >0$ and by
$\mathcal{O}(M_{h_1}+M_H\log(N_{n})+(1+\varpi/p)N_{n}\log(N_{n}))$ when
$d-\zeta/\alpha =0$.  It is also obvious they can be bounded by $\mathcal{O}(N_{n})$ and
$\mathcal{O}(N_{n}\log(N_{n}))$, respectively, if $M_H \ll N_{n}$,  $M_{h_1}\le N_{n}$
and $\varpi/p\leq C$ are provided.
\end{proof}

\begin{remark}
Since we have a good enough initial solution $\widetilde{u}^{h_{k+1}}$
in the second step of Algorithm \ref{Cascadic_MCS_Alg},
solving the nonlinear eigenvalue problem (\ref{cascadic_correct_eig_exact})
always dose not need many nonlinear iteration times (always $\varpi \le 3$).
\end{remark}

\begin{corollary}\label{Error_Estimate_Final_Corollary}
Under the same conditions of Theorem \ref{estimate_number_iter}
and (\ref{Condition_m_bar}), if $Ch_n\le\gamma$,
we have the following estimate
\begin{equation}\label{final_error_u_bar}
\|u^{h_n}-\bar{u}_{h_n}\|_1 \le 2\gamma h_n.
\end{equation}
\end{corollary}

\vspace{1ex}
If we choose the conjugate gradient method as the smoothing operator, then $\alpha = 1$
and the computational work of Algorithm \ref{Cascadic_MCS_Alg} can be bounded by
$\mathcal{O}\big((1+\varpi/p)N_{n}+M_{h_1}+M_H\log(N_{n})\big)$ or $\mathcal{O}(N_{n})$
provided $M_H \ll N_{n}$, $M_{h_1}\leq N_{n}$ and $\varpi/p\leq C$ for both
 $d=2$ and $d=3$ when we choose $1 < \zeta < d$.

When the symmetric Gauss-Seidel, the SSOR, the damped Jacobi or the Richardson iteration acts as
the smoothing  operator,  we know $\alpha=1/2$. Then  the computational work of
Algorithm \ref{Cascadic_MCS_Alg} can be bounded by
 $\mathcal{O}\big((1+\varpi/p)N_{n}+M_{h_1}+M_H\log(N_{n})\big)$ ($\mathcal{O}(N_{n})$
provided $M_H \ll N_{n}$, $M_{h_1}\leq N_{n}$ and $\varpi/p\leq C$)
only for $d=3$  when we choose $1 < \zeta < 3/2$.
In the case  of $\alpha=1/2$  and $d=2$,  from Theorem \ref{estimate_number_iter}
and its proof, we can only choose $\zeta=1$ and then the final error has the estimate
$\|u^{h_{n}}-\bar{u}_{h_{n}}\|_1 \leq Ch_{n}(1+CH)^{|\log(h_{n})|}$
and the computational work can only be bounded by
$\mathcal{O}\big((1+\varpi/p)N_{n}\log(N_{n})+M_{h_1}+M_H\log(N_{n})\big)$
($\mathcal{O}(N_{n}\log(N_{n}))$ provided $M_H \ll N_{n}$, $M_{h_1}\leq N_{n}$
and $\varpi/p\leq C$).

\section{Numerical exapmle}
In this section, we give a numerical example to illustrate the
efficiency of the cascadic multigrid scheme (Algorithm \ref{Cascadic_MCS_Alg})
 proposed in this paper. Here, we choose the conjugate-gradient iteration as the
smoothing operator ($\alpha=1$) and the number of iteration steps by
\begin{eqnarray*}
m_k = \lceil\bar{m}\sigma\beta^{\zeta(n-k)}\rceil\ \ \ {\rm for}\  k=2, \cdots, n
\end{eqnarray*}
with $\bar{m}=2$, $\sigma=2$, $\beta=2$, $\zeta=1.8$ and $\lceil r\rceil$
denoting the smallest integer which is not less than $r$

Here we give the numerical results of the cascadic multigrid
scheme for GPE problem on the two dimensional domain
$\Omega=(0,1 )\times (0, 1)$ with $W = x_1^2 + x_2^2$ and $\zeta=1$.  The sequence of
finite element spaces are constructed by
using linear element on the series of meshes which are produced by the 
regular refinement with $\beta =2$ (connecting the midpoints of each edge).
In this example, we use two meshes which are generated by Delaunay method as the initial mesh
$\mathcal{T}_{h_1}$ and set $\mathcal{T}_H=\mathcal{T}_{h_1}$
to investigate the convergence behaviors.
Figure \ref{Initial_Mesh_Delaunay} shows the corresponding
initial meshes: one is coarse and the other is fine.

Algorithm \ref{Cascadic_MCS_Alg} is applied to solve the GPE problem.
For comparison, we also solve the GPE problem by the direct finite element method.
From the error estimate result of GPEs by the finite element method, we have
\begin{equation*}
\delta_h(u)\approx h, \ \ \  \eta_a(V_h)\approx h.
\end{equation*}
Then from Corollary \ref{Error_Estimate_Final_Corollary}, the following estimates hold
\begin{equation*}
\|\bar{u}_{h_n} - u^{h_n}\|_1 \le Ch_n, \ \ \|\bar{u}_{h_n} - u^{h_n}\|_0
\le CHh_n, \ \ |\bar{\lambda}_{h_n} - \lambda^{h_n}| \le CHh_n.
\end{equation*}
We consider the Delaunay meshes (see Figure \ref{Initial_Mesh_Delaunay}).
\begin{figure}[htb]
\centering
\includegraphics[width=5cm,height=5cm]{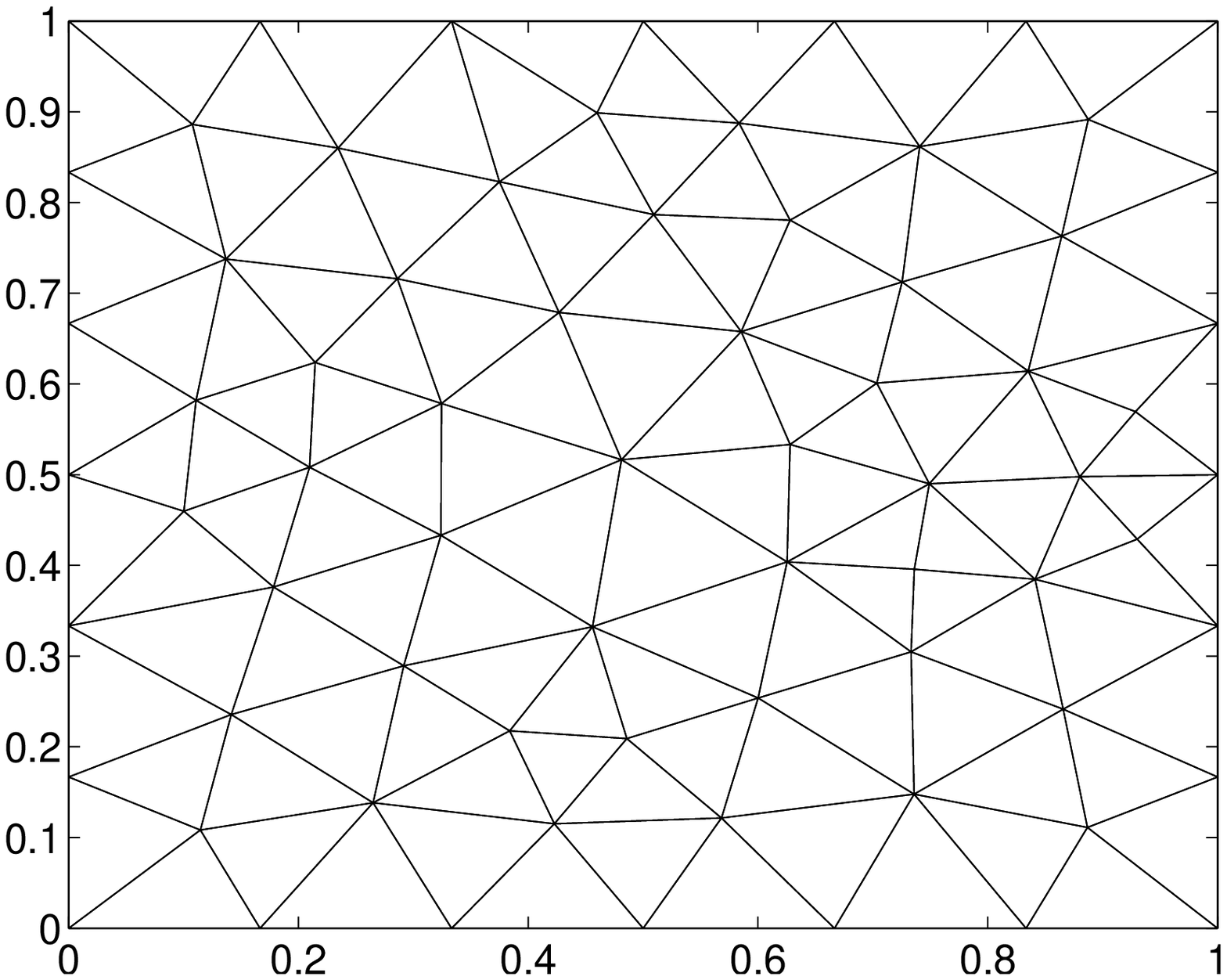}\ \ \
\includegraphics[width=5cm,height=5cm]{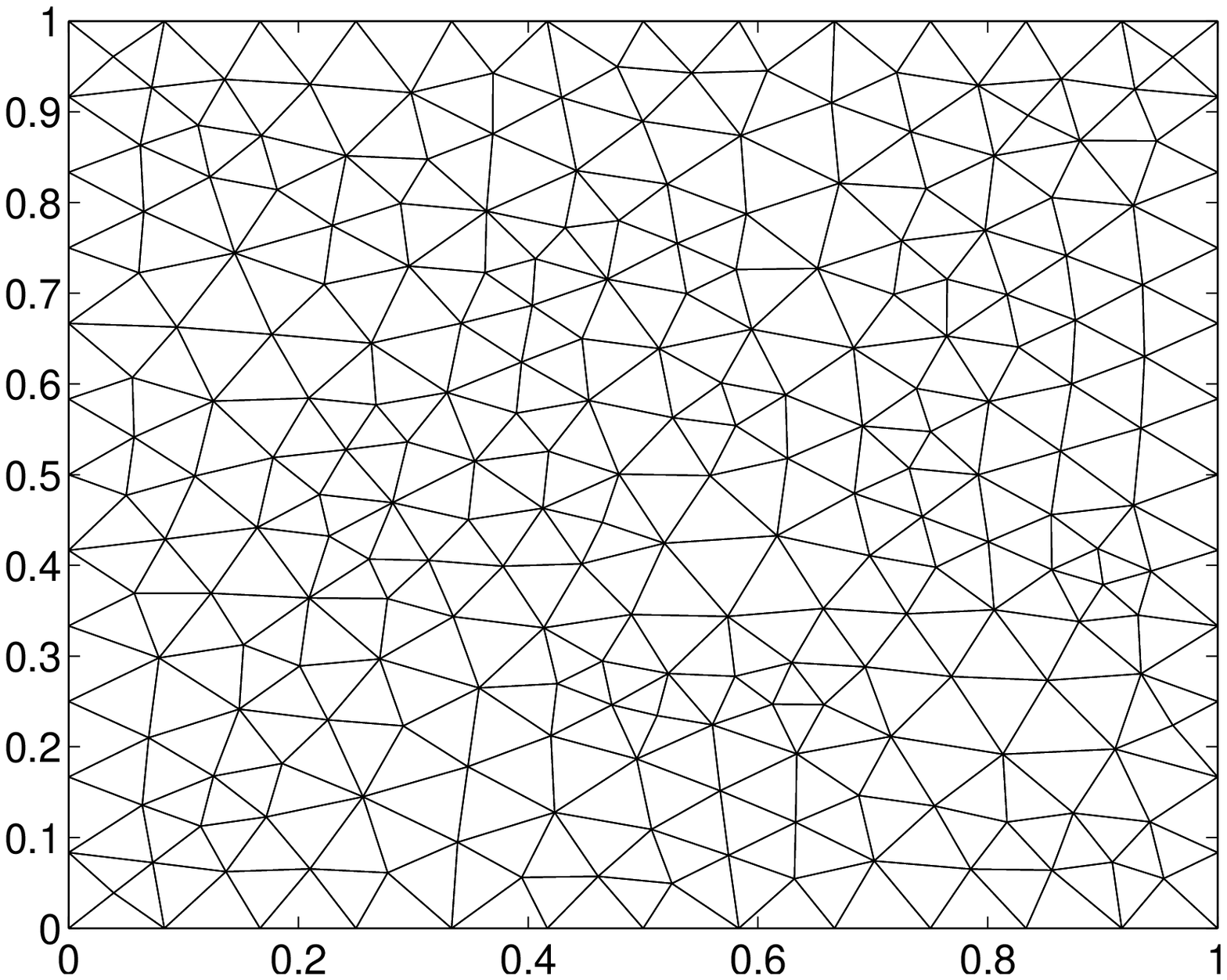}
\caption{\small\texttt The coarse and fine initial
meshes for the unit square (left: H=1/6 and right: H=1/12)}
\label{Initial_Mesh_Delaunay}
\end{figure}

Figure \ref{numerical_multi_grid_2D_Coarse_InitialMesh_Delaunay}
gives the corresponding numerical results for the GPE problem on the initial mesh
 illustrated by the left mesh in Figure \ref{Initial_Mesh_Delaunay}.
The corresponding numerical results for the GPE problem on the initial mesh
 illustrated by the right mesh in Figure \ref{Initial_Mesh_Delaunay}
 are shown in Figure \ref{numerical_multi_grid_2D_Fine_InitialMesh_Delaunay}.
\begin{figure}[htb]
\centering
\includegraphics[width=6cm,height=5cm]{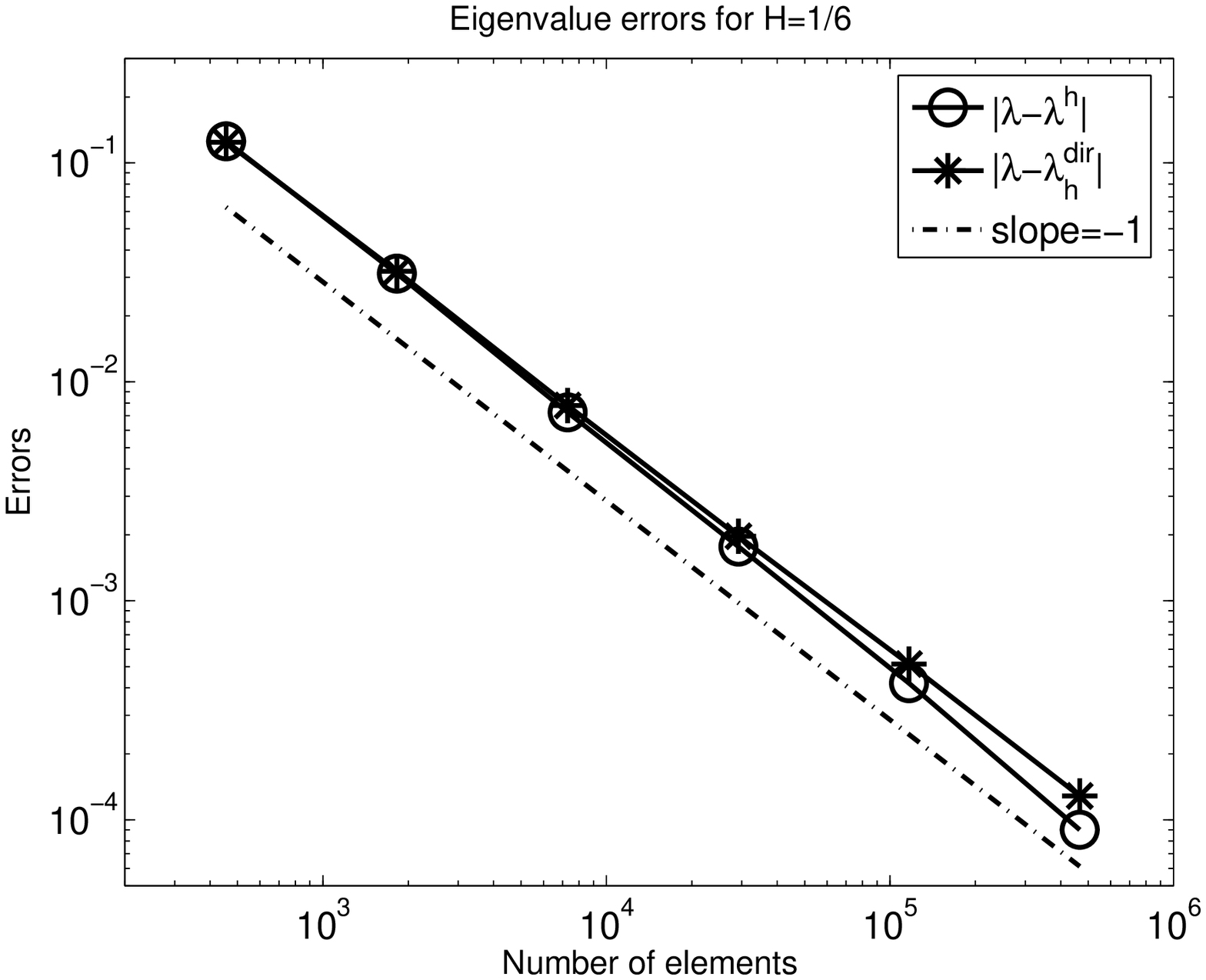}
\includegraphics[width=6cm,height=5cm]{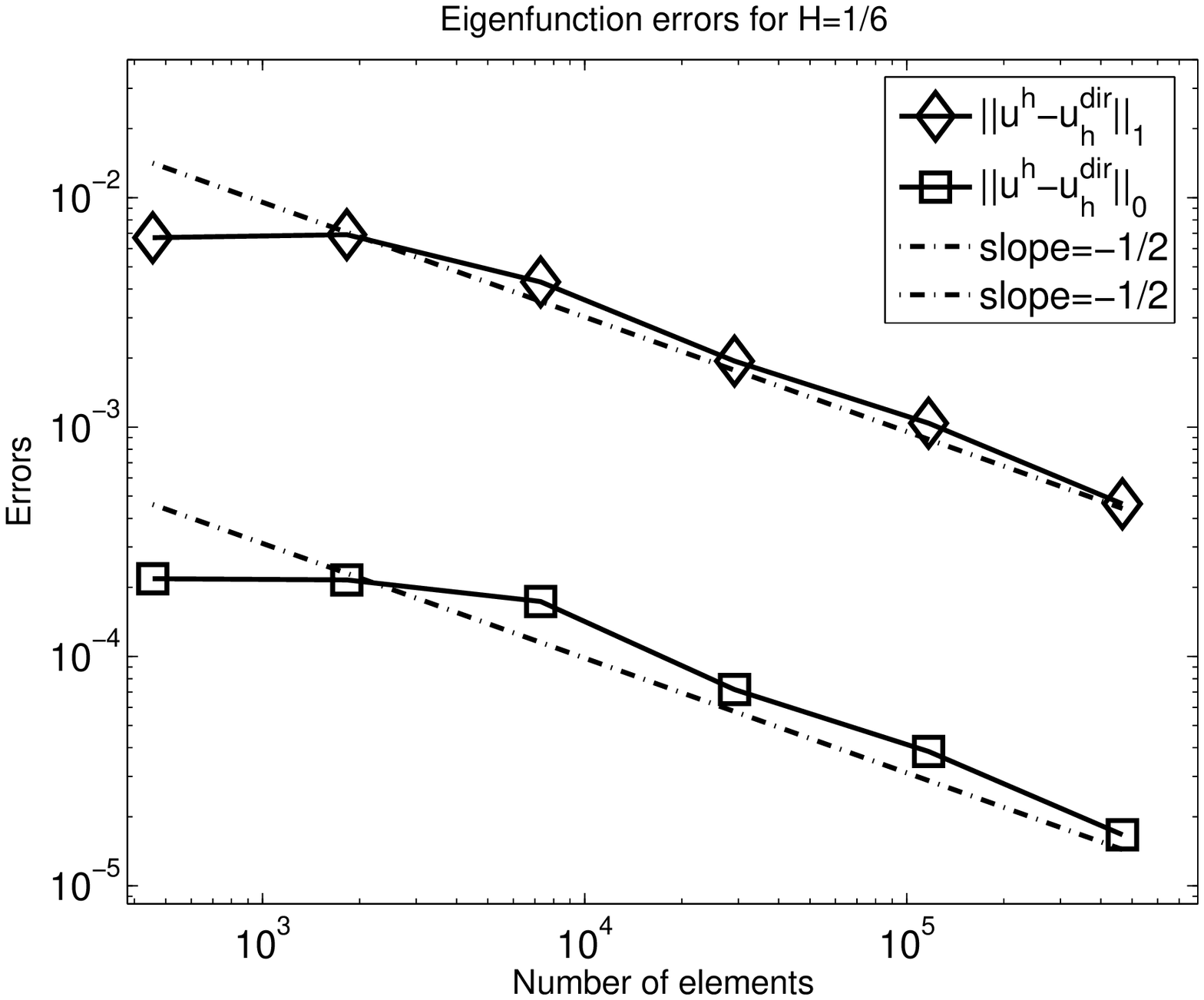}
\caption{\small\texttt The errors of the cascadic multigrid
algorithm for the GPE problem,
where $u^h$ and $\lambda^h$ denote the eigenfunction and eigenvalue
approximation by Algorithm \ref{Cascadic_MCS_Alg}, and $u_h^{\rm dir}$
and $\lambda_h^{\rm dir}$ denote the eigenfunction
and eigenvalue approximation by direct eigenvalue solving (The left
figure is the eigenvalue errors and the right figure is the eigenfunction errors which both
correspond to the left mesh in Figure \ref{Initial_Mesh_Delaunay}) }
\label{numerical_multi_grid_2D_Coarse_InitialMesh_Delaunay}
\end{figure}

\begin{figure}[htb]
\centering
\includegraphics[width=6cm,height=5cm]{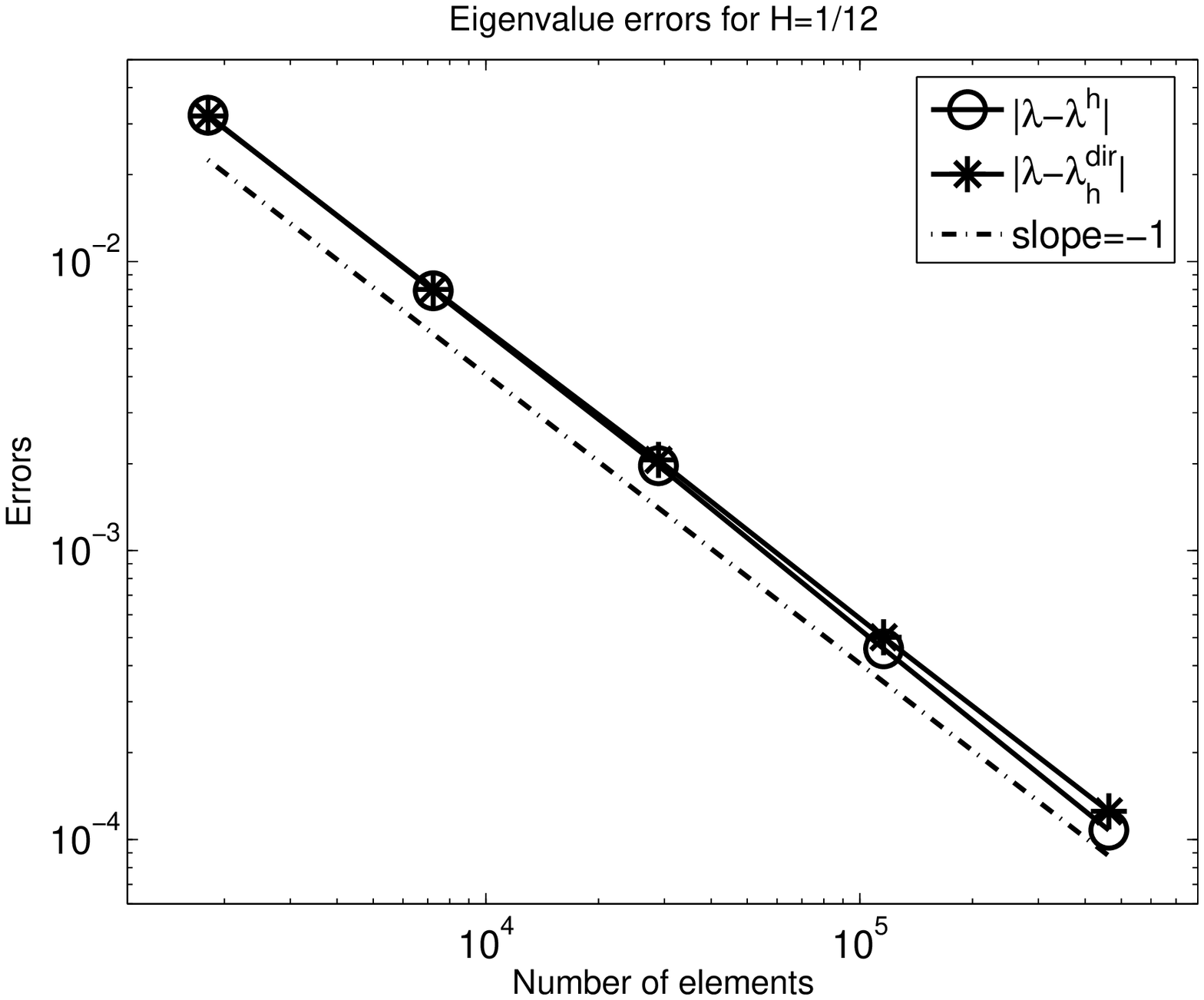}
\includegraphics[width=6cm,height=5cm]{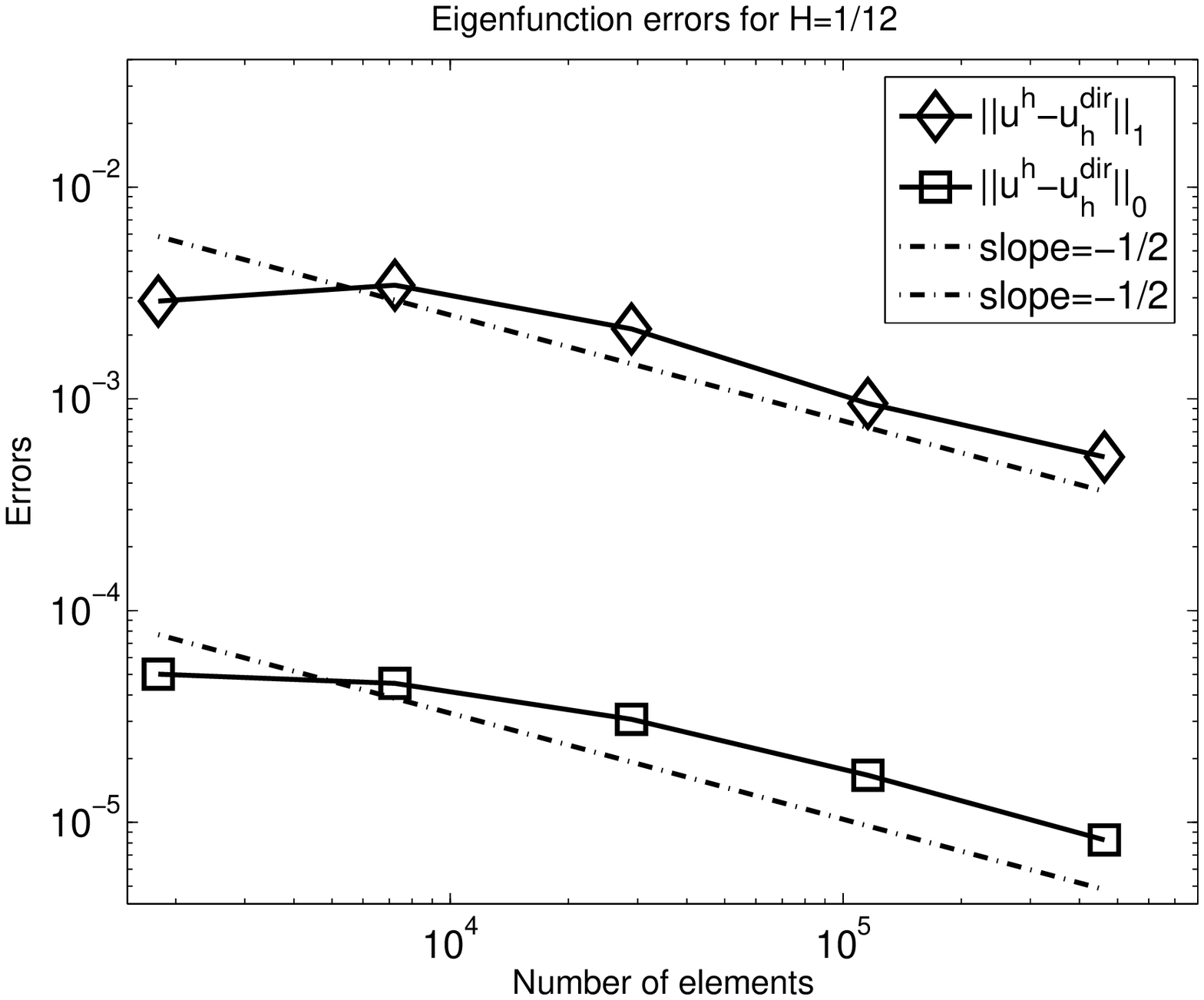}
\caption{\small\texttt The errors of the cascadic multigrid
algorithm for the GPE problem,
where $u^h$ and $\lambda^h$ denote the eigenfunction and eigenvalue
approximation by Algorithm \ref{Cascadic_MCS_Alg}, and $u_h^{\rm dir}$
and $\lambda_h^{\rm dir}$ denote the eigenfunction
and eigenvalue approximation by direct eigenvalue solving (The left
figure is the eigenvalue errors and the right figure is the eigenfunction errors which both
correspond to the right mesh in Figure \ref{Initial_Mesh_Delaunay}) }
\label{numerical_multi_grid_2D_Fine_InitialMesh_Delaunay}
\end{figure}

From Figures \ref{numerical_multi_grid_2D_Coarse_InitialMesh_Delaunay}
 and \ref{numerical_multi_grid_2D_Fine_InitialMesh_Delaunay},
we find the cascadic multigrid scheme can obtain
the same optimal error estimates as the direct eigenvalue
solving method for the eigenfunction approximations in the $H^1$-norm.

\begin{remark}
Note that by (\ref{lambda_bar_lambda_h}) and (\ref{final_error_u_bar}),
we {\bf do not} prove the optimal convergence rate for eigenvalue
error (i.e. $\big|\bar{\lambda}_{h_n} - \lambda^{h_n}\big|\le Ch_n^2$).
However, it is shown in the left of
Figures \ref{numerical_multi_grid_2D_Coarse_InitialMesh_Delaunay}
and \ref{numerical_multi_grid_2D_Fine_InitialMesh_Delaunay} that
$\big|\bar{\lambda}_{h_n} - \lambda^{h_n}\big|\le Ch_n^2$.
\end{remark}

\section{Concluding remarks}

In this paper, we present a type of cascadic multigrid method for GPE problem
based on the combination of the cascadic multigrid for boundary
value problems and the multilevel correction
scheme for eigenvalue problems. The optimality of the computational efficiency
has been demonstrated by theoretical analysis and numerical examples.

\bibliographystyle{plain}

\end{document}